\documentclass{article}
\usepackage{mathrsfs}
\usepackage{amsmath,amssymb,amsfonts,amsthm}
\usepackage{enumerate}
\usepackage{mathdots}
\usepackage{mathtools}
\usepackage{algorithm}
\usepackage{algpseudocode}
\usepackage{listings}
\usepackage{comment}
\usepackage{color}
\usepackage{pgfplotstable}
\usepackage{booktabs}
\usepackage{siunitx}
\definecolor{mygreen}{RGB}{28,172,0} 
\definecolor{mylilas}{RGB}{170,55,241}
\definecolor{mygray}{RGB}{230,230,230}
\definecolor{mygray2}{RGB}{200,200,200}

\renewcommand{\le}{\leq}
\renewcommand{\ge}{\geq}

\pgfplotstableset{
	every head row/.style={before row=\toprule,after row=\midrule},
	clear infinite
}

\DeclareMathOperator{\fl}{fl}

\newcommand{\w}{{_{_\mathcal W}}}

\newcommand{\qt}{{_{_\mathcal{QT}}}}
\newcommand{\toolbox}{\texttt{cqt-toolbox}}

\DeclareMathOperator*{\dotle}{\begin{array}{c}\cdot\\[-0.2cm] \le\end{array}}
\newtheorem{proposition}{Proposition}[section]
\newtheorem{lemma}[proposition]{Lemma}

\title{Quasi-Toeplitz matrix arithmetic: a MATLAB toolbox\thanks{This work has been supported by the GNCS/INdAM project 2018 ``Tecniche innovative
		per problemi di algebra lineare''. The authors are
  members of the research group GNCS.}}

\author{%
	Dario A. Bini\thanks{Dipartimento di Matematica, Pisa, Italy,
		{\tt bini@dm.unipi.it}} 
	\and
	Stefano Massei\thanks{EPF Lausanne, Switzerland,
		{\tt stefano.massei@epfl.ch}} 
	\and
	Leonardo Robol\thanks{Institute of Information Science and
		Technologies ``A. Faedo'', CNR, Pisa, Italy, 
		{\tt leonardo.robol@isti.cnr.it}}
}

\usepackage{amsmath,amssymb}
\usepackage{cite}
\usepackage{mathtools}
\usepackage{tikz}
\usepackage{pgfplots}
\usepackage{geometry}
\usepackage{hyperref}
\usepackage{color}

\newtheorem{theorem}[proposition]{Theorem}
\newtheorem{corollary}[proposition]{Corollary}

\theoremstyle{definition}
\newtheorem{definition}[proposition]{Definition}

\theoremstyle{remark}
\newtheorem{remark}[proposition]{Remark}

\usepackage{amsopn}
\DeclarePairedDelimiter{\norm}{\lVert}{\rVert}

\newcommand{\matlab}{\texttt{MATLAB}}

\renewcommand{\leq}{\leqslant}
\renewcommand{\geq}{\geqslant}
\renewcommand{\tilde}{\widetilde}
\begin{document}
\maketitle

\begin{abstract}
	A Quasi Toeplitz (QT) matrix is a semi-infinite matrix of the
	kind $A=T(a)+E$ where $T(a)=(a_{j-i})_{i,j\in\mathbb Z^+}$,
	$E=(e_{i,j})_{i,j\in\mathbb Z^+}$ is compact and the norms
	$\|a\|\w=\sum_{i\in\mathbb Z}|a_i|$ and $\|E\|_2$ are finite.
	These properties allow to approximate any QT-matrix,
	within any given precision, by means of a finite number of
	parameters.
	
	QT-matrices, equipped with the norm
	$\|A\|\qt=\alpha\|a\|\w+\|E\|_2$, for $\alpha = (1+\sqrt
	5)/2$, are a Banach algebra with the standard arithmetic
	operations. We provide an algorithmic description of these
	operations on the finite parametrization of QT-matrices, 
	and we develop a \matlab\ toolbox implementing them in a transparent way. 
	The toolbox is then extended to perform arithmetic operations on matrices
	of finite size that have a Toeplitz plus low-rank structure. This enables
	the development of algorithms for Toeplitz and quasi-Toeplitz matrices
	whose cost does not necessarily increase with the dimension of the problem. 
	
	Some examples of applications to computing matrix functions and 
	to solving matrix equations	are presented, and confirm the effectiveness
	of the approach.
\end{abstract}

\section{Introduction}

Toeplitz matrices, i.e., matrices having constant entries along their diagonals, are found in diverse settings of applied mathematics, ranging
from imaging to Markov chains, and from finance to the solution of PDEs. 
These matrices can be of large size, and often they are infinite
or semi-infinite in the original mathematical model. 

As shown in \cite{bottcher2012introduction}, semi-infinite Toeplitz matrices
do not form an algebra; in particular, neither product nor inverses of 
semi-infinite Toeplitz matrices are still Toeplitz structured in general. However, 
this property continues to hold up to a compact operator from $\ell^2$ onto itself, where $\ell^2$ is the linear space formed by sequences $x=(x_i)_{i>0}$ such that $\|x\|_2:=(\sum_{i=1}^{+\infty} |x_i |^2)^{1/2}<+\infty$. More
precisely, the set of semi-infinite Toeplitz matrices plus a compact
$\ell^2$ operator is a Banach algebra, that is, 
a Banach space with the $\ell^2$ operator norm, closed under matrix multiplication, where the associated operator norm is sub-multiplicative.
We refer to such matrices as 
\emph{Quasi-Toeplitz matrices}, in short QT matrices. Their computational properties have
been investigated in \cite{bini2016semi,bini2016aqt2,biniquadratic,bm:exp}. 

We provide a description of finitely representable QT matrices, together with the analysis of the computational properties of their arithmetic, moreover we provide an implementation of QT matrices in
the form of a \matlab\ toolbox called {\tt cqt-toolbox}
(fully compatible with GNU/Octave), where the
acronym {\tt cqt} stands for ``Computing with Quasi Toeplitz matrices'',
and show some examples of applications. The toolbox can be downloaded from
\url{https://github.com/numpi/cqt-toolbox}. 

\subsection{Motivation}
Matrices of infinite size are encountered in several 
applications which describe the behavior of systems with
a countable number of states, and more generally whenever
infinite dimensional objects are involved. Typical
examples come from queuing models where the number of states of the
stochastic process is infinitely countable, say, it can be
represented by the set $\mathbb Z$ of relative integers or by the set
$\mathbb Z^+$ of positive integers, so that the probability transition
matrix is bi-infinite or semi-infinite, respectively. In other models,
like the random walk in the quarter plane \cite{miyazawa2011light},
\cite{kobayashi2012revisit}, in the QBD processes \cite{lr:book}, and in
the more general MG1 and GM1 queues \cite{neuts81}, the set of states
is discrete and bidimensional, i.e., defined by integer pairs $(i,j)$ where at
least one component ranges in an infinite set.  Sometimes, these pairs
belong to $\mathbb Z\times\mathbb Z^+$ or to $\mathbb Z^+\times\mathbb
Z^+$.  In these cases, the probability transition matrix has a block
structure with infinitely many blocks and with blocks which have
infinite size.

A typical feature shared by many models is that --- sufficiently far form the border --- the transitions from a state to another
depend on their relative positions and are independent of
the single state, see for instance the tandem Jackson queue
\cite{jackson1957networks} or the random walk in the quarter plane
analyzed in \cite{miyazawa2011light ,kobayashi2012revisit}. In
these situations, the transition probability matrix is block Toeplitz almost everywhere
and  its blocks are Toeplitz except for some elements in the upper left corner. In particular, the blocks
can be written in the form $ T(a)+E $, where $T(a)=(a_{j-i})$ is the
Toeplitz matrix associated with the sequence $a=\{a_i\}_{i\in\mathbb
	Z}$, while $E$ is a matrix having only a finite number of nonzero
entries containing the information concerning the boundary
conditions. The computation of interesting quantities related to these models, e.g., the steady state vector, requires to solve quadratic matrix equations whose coefficients are given by the blocks of the transition probability matrix. 

The numerical treatment of problems involving infinite matrices is usually performed by
truncating the size to a finite large value, by
solving the finite problem obtained this way and using this finite
solution to approximate part of the solution of the infinite problem.

In \cite{lindner2006infinite} the author analyzes this approach
--- called the \emph{finite section method} --- for infinite linear systems, providing conditions that ensure
the solution of the truncated system to converge to the solution of the
infinite one, as the size of the section tends to $+\infty$.  The analogous strategy can be adopted for solving matrix equations or computing
matrix functions, but --- in general --- there is no guarantee of
convergence. In fact, in
\cite{latouche2002truncation,bean2010approximations,latouche2011queues}
bad effects of truncation are highlighted when solving infinite
quadratic matrix equations arising in the Markov chains framework.
In \cite{kapo} a method is designed for a subclass
of bidimensional random walks where the solution can be represented
in a special form. In particular, the authors point out the
difficulty to apply the matrix geometric method of Marcel Neuts
\cite{neuts81}, and therefore of solving a quadratic matrix equation,
due to the infinite size of the matrix coefficients and of the
solution.

Recently, a different approach has been introduced by studying
structures that allow finitely represented approximations of infinite matrices and that are preserved by matrix operations. Working with this kind of structured matrices does not require to truncate to finite size in order to carry out computations.

In \cite{bini2016semi,bini2016aqt2,biniquadratic, bm:exp}, the class $\mathcal{QT}$ of
semi-infinite \emph{Quasi-Toeplitz} (QT) matrices has been
introduced. This set is formed by matrices of the kind $A=T(a)+E$
where, in general, $a(z)=\sum_{i\in\mathbb Z} a_iz^i$ is a Laurent series such that
$\|a\|\w=\sum_{i=-\infty}^{+\infty}|a_i|$ is finite, and $E$ is a
compact correction.
Each element of this class can be approximated --- at any
arbitrary precision --- with the sum of a banded Toeplitz $T(\widetilde
a)$ plus a matrix $\widetilde E$ with finite support. 
QT-matrices form a Banach algebra and enable the implementation of an approximate
matrix arithmetic which operates on the elements of the class. Using this
tool, one can deal with certain classical linear algebra issues just
plugging the new arithmetic into the procedures designed for matrices of
finite size.

Another intriguing aspect of QT-arithmetic is that it can be easily adapted to finite matrices of the form Toeplitz plus low-rank. This paves the way for efficiently computing functions of Toeplitz matrices, which has been recently raised some attention. See \cite{gutierrez2007functions} for applications concerning signal processing. In \cite{lee2010shift,kressner2016fast}
computing the matrix exponential
of large Toeplitz matrices is required for option pricing with the
Merton model. 

\subsection{New contributions}
In this paper, by continuing the work started in
\cite{bini2016semi,bini2016aqt2,biniquadratic,bm:exp}, we  analyze
the representation of QT matrices by means of a finite number of parameters, in a sort of analogy with the finite floating point representation of real numbers. Moreover, we investigate some computational issues related to the definition and the implementation of a matrix arithmetic in this class. 
Finally, we provide an effective implementation of the class of finitely
representable QT matrices together with the related matrix arithmetic
in the MATLAB toolbox \toolbox.

In order to perform approximations of QT matrices with finitely representable matrices, we introduce  the following norm
\[
\|A\|\qt=\alpha\|a\|\w+\|E\|_2,\quad \alpha=\frac{1+\sqrt 5}{2}.
\]
This norm is different from the one used in 
\cite{bini2016semi,bini2016aqt2,biniquadratic,bm:exp}:
it is slightly more general, and still makes the set $\mathcal{QT}$ a Banach algebra. It can be shown 
that any value of $\alpha \geq \frac{1+\sqrt{5}}{2}$ would make
this set a Banach algebra. Moreover, we will see that this choice allows a complete control on the approximation errors and enables us to perform, in a safe way, different computational operations like compression or matrix inversion.

The paper is organized as follows. In Section~\ref{sec:qt} we recall
the definition and some theoretical results about QT-matrices,
together with the norm $\|\cdot\|\qt$. We introduce the class of
finitely representable QT matrices and provide a first description of
the {\tt cqt-toolbox}.

Section \ref{sec:arith} deals with the definition and the analysis of the
arithmetic operations in the algebra of finitely representable QT
matrices. The first subsections deal with addition, multiplication, inversion and
compression. 
Then, Section \ref{sec:finite} describes the
extension of the arithmetic (and of the toolbox) to the case of finite QT matrices.
Section~\ref{sec:app} provides some examples of applications,
Section \ref{sec:conclusion} draws the conclusions.

In the appendices,  we provide some
details on the Sieveking-Kung algorithm for triangular Toeplitz matrix
inversion \ref{sec:sk}, and on the main algorithms for computing the Wiener-Hopf
factorization \ref{sec:wh}.

\section{The class of QT matrices}\label{sec:qt}

We start by introducing the set of semi-infinite matrices that we are going to implement,
 recall its main properties and provide an effective (approximate)
\emph{finite} representation.

\subsection{The Wiener class and semi-infinite Toeplitz matrices}
We indicate with $\mathbb T := \{
z\in\mathbb C:\ |z|=1\}$ the complex unit circle, and with $\mathcal W$ the Wiener class
formed by the functions $a(z)=\sum_{i=-\infty}^{+\infty}a_iz^i:\mathbb
T\to\mathbb C$ such that $\sum_{i=-\infty}^{+\infty}|a_i|<+\infty$,
that is functions expressed by a Laurent series with absolutely
summable coefficients.

The set $\mathcal W$, endowed with the norm
$\|a\|\w:=\sum_{i\in\mathbb Z}|a_i|$, is a Banach algebra.  By the
Wiener theorem, \cite[Section 1.4]{bottcher2005spectral}, a Laurent
series in $\mathcal W$ is invertible if and only if $a(z) \neq 0$ on
the unit circle. Under the latter condition, there exist functions
$u(z)=\sum_{i=0}^\infty u_iz^i, l(z)=\sum_{i=0}^\infty
l_iz^i\in\mathcal W$ with $u(z),l(z)\ne 0$ for $|z|\le 1$ such that
the factorization
\[
a(z)=u(z)z^ml(z^{-1}), \qquad 
u(z) = \sum_{i=0}^\infty u_iz^i, \qquad 
l(z) = \sum_{i = 0}^\infty l_i z^i
\]
holds where $m$ is the winding number of $a(z)$.  The above
decomposition is known as \emph{Wiener-Hopf factorization}.  We refer the
reader to the first chapter of the book \cite{bottcher2005spectral}
for more details.

We associate an element $a(z)=\sum_{i\in\mathbb Z}a_iz^i$ of the 
Wiener class 
with the semi-infinite Toeplitz
matrix $T(a)=(t_{i,j})$ such that $t_{i,j}=a_{j-i}$ for
$i,j\in\mathbb Z^+$, or, in matrix form
\[
T(a)=\begin{bmatrix}
a_0&a_1&a_2&\dots\\
a_{-1}&a_0&a_1 &\ddots\\
a_{-2}& a_{-1}&a_0&\ddots\\
\vdots&\ddots&\ddots&\ddots
\end{bmatrix}.
\]

Matrices $T(a)$ naturally define operators from $\ell^2$ into itself,
where $\ell^2$ is the set of sequences $x=(x_i)_{i\ge 1}$ such that
$\|x\|_2=( \sum_{i=1}^\infty |x_i|^2)^{\frac 12}$ is finite. 
In particular, one can show that $\norm{T(a)}_2 \leq \norm{a}\w$, 
where $\norm{T(a)}_2$ denotes the operator norm induced by the
$\ell^2$-norm on the operator $T(a)$. 

Given $a(z)\in\mathcal W$ we denote $a^+(z)=\sum_{i=1}^\infty a_iz^i$,
$a^-(z)=\sum_{i=1}^\infty a_{-i}z^i$, so that we may write
$a(z)=a^-(z^{-1})+a_0+a^+(z)$. Moreover, given the power series
$b(z)=\sum_{i=0}^\infty b_iz^i\in\mathcal W$, we denote
$H(b)=(h_{i,j})$ the semi-infinite Hankel matrix defined by
$h_{i,j}=b_{i+j-1}$, for $i,j\in\mathbb Z^+$.

Despite $\mathcal W$ is closed under multiplication, the corresponding
matrix class formed by semi-infinite Toeplitz matrices of the kind
$T(a)$, for $a\in\mathcal W$, is not. However, it satisfies this
property up to a compact correction \cite{bottcher2005spectral} as
stated by the following result.

\begin{theorem}\label{thm1}
	Let $a(z),b(z)\in\mathcal W$ and set $c(z)=a(z)b(z)$. Then
	\[
	T(a)T(b)=T(c)-H(a^-)H(b^+).
	\]
	where $H(a^-)=(h^-_{i,j})_{i,j\ge 1}$, $H(b^+)=(h^+_{i,j})_{i,j\ge
		1}$ with $h^-_{i,j}=a_{-(i+j+1)}$ and $h^+_{i,j}=b_{i+j+1}$.
	Moreover, the matrices $H(a^-)$ and $H(b^+)$ define
		compact operators on $\ell^2$ and are such that
	$\|H(a^-)\|_2\le\|a^-\|\w$ and $\|H(b^+)\|_2\le\|b^+\|\w$.
\end{theorem}

Assume that $a(z)=\sum_{-n_-}^{n_+} a_iz^i$ where $0 \leq n_-,n_+ <
\infty$.  We recall that (see \cite{bottcher2012introduction},
\cite{Gohberg1952}) for a continuous symbol $a(z)$ the matrix $T(a)$
is invertible if and only if $a(z)\ne 0$ for $|z|=1$ and the winding
number of $a(z)$ is 0. On the other hand, from \cite[Theorem
1.14]{Gohberg1952} the latter condition implies that there exist
polynomials $u(z)=\sum_{i=0}^{n_+} u_iz^i$ and
$l(z)=\sum_{i=0}^{n_-}l_iz^{i}$ having zeros of modulus less than 1
such that $a(z)=u(z)l(z^{-1})$. Therefore we may conclude that if
$T(a)$ is invertible then there exists the \emph{Wiener-Hopf}
factorization $a(z)=u(z)l(z^{-1})$ so that $T(a)=T(u)T(l)^T$ and we
may write
\begin{equation}\label{eq:tinv}
T(a)^{-1}=(T(l)^T)^{-1}T(u)^{-1}.
\end{equation}
Observe that since $u(z)$ and $l(z)$ have zeros of modulus less than
1,  by Wiener's theorem, are invertible as power series. 
These arguments, together with Theorem \ref{thm1}, lead to the
following result \cite{bottcher2012introduction}.

\begin{theorem}\label{thm:inv}
	If $a(z)=\sum_{-n_-}^{n_+} a_iz^i$, then $T(a)$ is invertible
	in $\ell^2$ if and only if there exists the Wiener-Hopf
	factorization $a(z)=u(z)l(z^{-1})$, for $u(z)=\sum_{i=0}^{n_+}
	u_iz^i$ and $l(z)=\sum_{i=0}^{n_-}l_iz^{i}$ having zeros of
	modulus less than 1. Moreover $u^{-1}(z),l^{-1}(z)\in\mathcal
	W$ so that
	\begin{equation}\label{eq:tinvH}
	\begin{split}
	&a^{-1}(z)=l(z^{-1})^{-1}u(z)^{-1},\\
	&T(a)^{-1}=T(l^{-1})^TT(u^{-1})=T(a^{-1})+E,\quad E=H(l^{-1})H(u^{-1}),\\
	&\|E\|_2\le\|l^{-1}\|\w\|u^{-1}\|\w.
	\end{split}
	\end{equation}
\end{theorem}

\subsection{Quasi-Toeplitz matrices}
We are ready to introduce the central notion of this paper.

\begin{definition}\label{def:qt}
	We say that the semi-infinite matrix $A$ is  {\em Quasi-Toeplitz
		(QT)} if it can be written in the form
	\[
	A=T(a)+E,
	\]
	where 
	$a(z)=\sum_{i=-\infty}^{+\infty} a_iz^i$ 
	is in the Wiener class,
	and $E=(e_{i,j})$
	defines a compact operator on $\ell^2$. 
\end{definition}

It is well known \cite{bottcher2012introduction} that  the
class of $QT$ matrices, equipped with the $\ell^2$ norm, is a Banach
algebra. However, the $\ell^2$ norm can be difficult
to compute numerically, so we prefer to
introduce a slightly different norm which still preserves the Banach
algebra property. Let $\alpha=(1+\sqrt 5)/2$ and set
$\|A\|\qt=\alpha\|a\|\w+\|E\|_2$. Clearly, $\|A\|\qt$ is a norm which
makes complete the linear space of QT-matrices. Moreover, it is easy
to verify that this norm is sub-multiplicative, that is,
$\|AB\|\qt\le\|A\|\qt\|B\|\qt$ for any pair of QT matrices $A,B$. This
way, the linear space of QT matrices endowed with the norm
$\|\cdot\|\qt$ forms a Banach algebra that we denote by
$\mathcal{QT}$. Observe also that $\|A\|_2\le \|A\|\qt$ for any QT
matrix $A$.

The next lemma ensures that every QT matrix admits finitely representable
approximations with arbitrary accuracy.

\begin{lemma} \label{lem:truncation-wiener}
	Let $A=T(a)+E\in\mathcal{QT}$ and $\epsilon > 0$. Then, there
	exist non negative integers $n_-,n_+,n_r,n_c$ such that the matrix
	$\widehat A=T(\widehat a)+\widehat E$, defined by
	\[\hat a(z) = \sum_{i = -n_-}^{n_+} a_i z^i,\qquad
	\widehat E_{ij} = \begin{cases}
	E_{ij} & \text{if} \quad 1 \leq i\leq n_r\quad\text{and}\quad 1\leq j \leq n_c\\
	0 & \text{otherwise}
	\end{cases},
	\] 
	verifies
	$\norm{A-\widehat A}_{\mathcal{QT}} \leq \norm{A}_{\mathcal{QT}} \cdot \epsilon$. 
\end{lemma}
\begin{proof}
	Since $A\in\mathcal{QT}$ then $\norm{a}\w=\sum_{j\in\mathbb
		Z}|a_j|<\infty$. This means that there exist $n_-,n_+$ such
	that
	\begin{equation}
	\label{ineq:tmp}	
	\|a-\widehat a\|\w=\sum_{j<-n_-}|a_j|+\sum_{j>-n_+}|a_j|\leq \epsilon\|A\|\qt /\alpha. 
	\end{equation}
	Since $E$ represents a compact operator, there exist
	$k\in\mathbb N$, $\sigma_i\in\mathbb R^+$ and
	$u_i,v_i\in\mathbb R^{\mathbb N}$ with unit $2$-norm,
	$i=1,\dots,k$, which verify
	$\norm{E-\sum_{i=1}^k\sigma_iu_iv_i^T}_2\leq
	\frac{\epsilon}2\norm{E}_2$. The condition
	$\norm{u_i}_2=\norm{v_i}_2=1$ implies that there exist two
	integers $n_r$ and $n_c$ such that the vectors
	\[
	\widetilde u_i(j)=\begin{cases}
	u_i(j)&\text{if }j> n_r\\
	0&\text{otherwise}
	\end{cases},\qquad
	\widetilde v_i(j)=\begin{cases}
	v_i(j)&\text{if }j> n_c\\
	0&\text{otherwise}
	\end{cases},
	\]
	have $2$-norms bounded by
	$\frac{\epsilon\norm{A}\qt}{4k\max_i\sigma_i}$. Then, denoting
	by $\widehat u_i:=u_i-\tilde u_i$ and $\widehat
	v_i:=v_i-\tilde v_i$, and setting $\widehat E:=
	\sum_{i=1}^k\sigma_i\widehat u_i\widehat v_i^T$, we find that
	\[
	\norm{\widehat u_i\widehat v_i^T-u_iv_i^T}_2=\norm{\tilde u_i
		v_i^T+ \widehat u_i \tilde v_i^T}_2\leq
	\frac{\epsilon\norm{A}_\qt}{2k\max_i\sigma_i}\quad
	\Longrightarrow\quad \norm{\widehat
		E-\sum_{i=1}^k\sigma_iu_iv_i^T}_2\leq
	\frac{\epsilon}{2}\norm{A}_\qt.
	\]
	To conclude, we have $\norm{E-\widehat E}_2\leq
	\norm{E-\sum_{i=1}^k\sigma_iu_iv_i^T}_2 +
	\norm{\sum_{i=1}^k\sigma_iu_iv_i^T-\widehat E}_2\leq
	\epsilon\norm{A}_\qt$. Thus, from the latter inequality and
	from \eqref{ineq:tmp} we get$\|A-\widehat
	A\|\qt=\alpha\|a-\hat a\|\w+\|E-\widehat E\|_2\le
	\epsilon\|A\|\qt$.
\end{proof}

This result makes it possible to draw an analogy between the
representation of semi-infinite quasi Toeplitz matrices and floating
point numbers. When representing a real number $a$ in floating point
format fl$(a)$ it is guaranteed that
\[
\fl(a ) = a + \mathcal E, \qquad 
|\mathcal E| \leq |a| \cdot \epsilon, 
\]
where $\epsilon$ is the 
so-called unit roundoff. 

We design a similar framework for QT-matrices. More precisely,
analogously to the operator ``fl$(\cdot)$'', we introduce a
``truncation'' operator $\mathcal{QT}(\cdot)$\ that works separately
on the Toeplitz and on the compact correction, as described by
Lemma~\ref{lem:truncation-wiener}.  So, for a QT-matrix $A=T(a) +
E_a$, we have
\begin{equation} \label{eq:floating-point-qt}
\mathcal{QT}(A) = T(\widehat a)+ \widehat E_a =A  + \mathcal E, \qquad 
\norm{\mathcal E} \leq \norm{A }_{\mathcal QT} \cdot \epsilon, 
\end{equation}
where $\epsilon$ is some prescribed tolerance set a priori
(analogously to the unit roundoff), and $\mathcal{QT}(A)$ is given by
the sum of a banded Toeplitz matrix $T(\widehat a)$ and a semi
infinite matrix $\widehat E_a$, with finite support.

Matrices of the kind $\mathcal{QT}(A)$ form the class of {\em
	finitely} representable Quasi Toeplitz matrices, where, unlike the case
of floating point numbers, the lengths of the representations are not
constant and may vary  in order to guarantee a
uniform bound to the relative error in norm.

The \toolbox\
collects tools for operating with finitely representable Quasi
Toeplitz matrices. The Toeplitz part is stored into two vectors
containing the coefficients of the symbol with non positive and with
non negative indices, respectively. The compact correction is
represented in terms of two matrices $\widehat U_a\in\mathbb
R^{n_r\times k}$ and $\widehat V_a\in\mathbb R^{n_c\times k}$ such
that $\widehat E_a(1:n_r,1:n_c)=\widehat U_a\widehat V_a^T$ coincides
with the upper left corner of the correction.

In order to define a new finitely representable $\mathcal{QT}$ matrix, one
has to call the $\texttt{cqt}$ constructor, in the following way:

\begin{lstlisting}
>> A = cqt(neg, pos, E);
\end{lstlisting}

In the above command, the vectors \lstinline-pos- and \lstinline-neg-
contain the coefficients of the symbol $a(z)$ with non positive and
non negative indices, respectively, and \lstinline-E- is a finite section
of the correction representing its non zero part. For example, to define a matrix $A = T(a) + E$ as
follows
\[
A =
\begin{bmatrix}
1 & 2 & 1 \\
-3 & 4 & 1 & 1 \\
& -1 & 2 & 1 & 1 \\
& & \ddots & \ddots & \ddots & \ddots \\
\end{bmatrix} = T(a) + E, \qquad
\begin{cases}    
a(z) = - z^{-1} + 2 + z + z^2, \\    
E = \begin{bmatrix}
-1&1\\
-2&2
\end{bmatrix}    
\end{cases}
\]
one needs to type the following \matlab\ commands:
\begin{lstlisting}
>> E = [-1, 1;-2, 2];
>> pos = [2 1 1];
>> neg = [2 -1];
>> A = cqt(neg, pos, E);
\end{lstlisting}
Notice that the constant coefficient is contains in both vectors, \lstinline-pos-
and \lstinline-neg-. If the user supplies two different values, the toolbox
returns an error. 
It is also possible to specify the correction in the factorized form $E=UV^T$.
\begin{lstlisting}
>> U = [1; 2];
>> V = [-1; 1];
>> A = cqt(neg, pos, U, V);
\end{lstlisting}
Removing the \lstinline-;- from the last expression will cause \matlab\
to print a brief description of the infinite matrix.
\begin{lstlisting}
>> A

A = 

CQT Matrix of size Inf x Inf

Rank of top-left correction: 1

- Toeplitz part (leading 4 x 5 block): 
 2     1     1     0     0
-1     2     1     1     0
 0    -1     2     1     1
 0     0    -1     2     1

- Finite correction (top-left corner): 
-1     1
-2     2
\end{lstlisting}

The different parts composing a $\mathcal{QT}$ matrix $A$ can be fetched
independently using the instructions \lstinline-symbol- and \lstinline-correction-.
For the previous example we have:

\begin{lstlisting}
>> [neg, pos] = symbol(A)

neg =

2    -1

pos =

2     1     1

>> E = correction(A)

E =

-1     1
-2     2
\end{lstlisting}
The command \lstinline|[U, V] = correction(A)| allows to retrieve the
correction in factorized form. The rank of the latter can be obtained with 
the command \lstinline-cqtrank-. 

\section{Arithmetic operations}\label{sec:arith}
When performing floating point operations
it is guaranteed
that 
\[
\fl(a \odot b) = a \odot b + \mathcal E, \qquad 
|\mathcal E| \leq (a \odot b) \cdot \epsilon, 
\]
where $\odot$ is any basilar arithmetic operation (sum, subtraction,
multiplication, and division). 

Extending the analogy, the matrix arithmetic in the set of finitely representable QT matrices
is implemented in  a way that the outcome of $A \odot B$,
for any pair of finitely representable $A,B\in\mathcal{QT}$ and
$\odot \in \{$\lstinline:+,-,*,/,\:$\}$, is represented by
$\mathcal{QT}(A\odot B)$ such that
\begin{equation} \label{eq:floating-point-arith}
A\odot B =\mathcal{QT}(A\odot B)+\mathcal E,\quad
\|\mathcal E\|\qt\le\epsilon\|A\odot B\|\qt.
\end{equation}

Notice that the outcome of an
arithmetic operation between finitely representable QT matrices, might not
be finitely representable, that is why we need to apply the $\mathcal{QT}(\cdot)$
operator on it. 

Another benefit of the $\mathcal{QT}(\cdot)$ operator is that it
optimizes the memory usage, since it minimizes the number of
parameters required to store the data up to the required accuracy.
The practical implementation of $\mathcal{QT}(\cdot)$ is reported in
Section~\ref{sec:compress}.

We now describe how the arithmetic operations are performed in the
\toolbox. These overloaded operators correspond to the built-in
functions of \matlab, i.e., they can be invoked with the usual operators
\lstinline:+,-,*,/,\:. Since we represent only the non zero sections of infinite
objects we rely on operations between matrices and vectors that might
be of non compatible sizes, e.g., sum of vectors with different
lengths. This has to be interpreted as filling the missing entries with
zeros.

\subsection{Addition}
Given two finitely representable QT matrices $A = T(a) + E_a$ and $B =
T(b) + E_b$, the matrix $C=A+B$ is defined by the symbol
$c(z)=a(z)+b(z)$ and by the correction $E_c=E_a+E_b$. Hence, the symbol
$c(z)$ is computed with two sums of vectors. The factorization
$E_c=U_cV_c^T$ is given by

\begin{equation}\label{eq:add}
U_c=[U_a,U_b], \quad V_c=[V_a,V_b].
\end{equation}

Then, applying the  compression technique, where $U_c$ and $V_c$ are replaced by matrices $\widehat U_c$ and $\widehat V_c$, respectively, having a lower number of columns and such that $\|E_c-\widehat U_c\widehat V_c^T\|_2$ is sufficiently small, we get

\[
\mathcal{QT}(A+B)=A+B+\mathcal E, 
\quad \|\mathcal E\|\qt\le\epsilon\|A+B\|\qt.
\]
The compression technique will be described in Section \ref{sec:compress}.

We refer to $\mathcal E$ as the local error of the
addition. Observe that if the operands $A$ and $B$ are affected
themselves by an error ${\mathcal E}_A$ and ${\mathcal E}_B$,
respectively, that is, the original QT-matrices $A$ and $B$ are
represented by approximations $\widehat A$ and $\widehat B$,
respectively such that
\begin{equation}\label{eq:errin}
\widehat A = A+{\mathcal E}_A, \qquad 
\widehat B = B+{\mathcal E}_B,
\end{equation}
then the computed matrix $\mathcal{QT}(\widehat A+\widehat B)$ differs
from $A+B$ by the total error given by
\begin{equation}\label{eq:err_add}
{\mathcal QT}(\widehat A+\widehat B)-(A+B)={\mathcal E}_A+{\mathcal
	E}_ B +{\mathcal E},
\end{equation}
where ${\mathcal E}_A+{\mathcal E}_ B$ is the inherent error caused by
the approximated input, while $\mathcal E$ is the local error.
Equation \eqref{eq:err_add} says that the global error is the sum of
the local error and the inherent error, and can be used to perform
error analysis in the QT-matrix arithmetic.
\subsection{Multiplication}
In view of Theorem \ref{thm1} we may write 
\[
C=AB = T(c)-H(a^-)H(b^+)+T(a)E_b+E_aT(b)+E_aE_b = T(c)+E_c,
\]
where $c(z)=a(z)b(z)$ and
\[
E_c=T(a)E_b+E_aT(b)+E_aE_b-H(a^-)H(b^+).
\]
The symbol $c(z)$ is obtained by computing the convolution of the
vectors representing the symbols $a(z)$ and $b(z)$, respectively.

For the correction part, denoting by $E_a=U_aV_a^T$, $E_b=U_bV_b^T$, $H(a^-)=M_a N_a^T$,
$H(b^+)=M_b N_b^T$ the decompositions of the matrices involved, we may write $E_c=U_cV_c^T$ with
\[
U_c=\left[ T(a)U_b,~U_a,-M_a\right],\qquad
V_c=\left[V_b,~T(b)^TV_a+V_b(U_b^TV_a),  N_b(M_b^TN_a) \right].
\]
Notice that, the products $T(a)U_b$ and $T(b)^TV_a$ generate matrices
with infinite rows and finite support. The computation of the non zero
parts of the latter require only a finite section of $T(a)$ and
$T(b)^T$, respectively. These operations are carried out efficiently
relying on the \emph{fast Fourier transform} (FFT).

The compressed outcome ${\mathcal QT}(AB)$  satisfies the equation
\[
\mathcal{QT}(AB)=AB+\mathcal{E},
\quad \|\mathcal E\|\qt\le\epsilon\|AB\|\qt,
\]
where $\mathcal E$ is the local error of the operation. If the operands $A$ and $B$ are affected by errors $\mathcal{E}_A$
and $\mathcal{E}_B$, respectively, such that \eqref{eq:errin} holds,
then  the global error in the computed product is given by
\begin{equation}\label{eq:err_mul}
\mathcal{QT}(\widehat A\widehat
B)-AB=\mathcal{E}+A\mathcal{E}_B+B\mathcal{E}_A+\mathcal{E}_A\mathcal{E}_B
\end{equation} 
where $A\mathcal{E}_B+B\mathcal{E}_A+\mathcal{E}_A\mathcal{E}_B$ is
the inherent error caused by the approximated input, while $\mathcal
E$ is the local error of the approximated multiplication.  In a first
order analysis we may replace the inherent error with
$A\mathcal{E}_B+B\mathcal{E}_A$ neglecting the quadratic part
$\mathcal{E}_A\mathcal{E}_B$.

\subsection{Inversion}\label{sec:inv}
Let $A=T(a)+E_a\in\mathcal{QT}$ be a finitely representable QT matrix such that the symbol 
\[
a(z)=\sum_{i=-n_a^-}^{n_a^+}a_iz^i
\]
admits the Wiener-Hopf factorization in the form $a(z)=u(z)l(z^{-1})$,
so that $T(a)=T(u)T(l)^T$ is invertible and $T(a)^{-1}=(T(l)^{-1})^T
T(u)^{-1}$. Assume also that $E_a$ is given in the factored form
$E_a=U_aV_a^T$ where $U_a$ and $V_a$ are matrices formed by $k$
columns and have null entries if the row index is greater than $m_a$.

Thus, we may write $A=T(a)(I+T(a)^{-1}E_a)$ so that, if $I+T(a)^{-1}E_a$ is
invertible then also $A$ is invertible and
\begin{equation}\label{eq:inv}
A^{-1}=(I+T(a)^{-1}E_a)^{-1}T(a)^{-1},\quad 
T(a)^{-1}=(T(l)^{-1})^T T(u)^{-1} .
\end{equation}
Observe also that by Theorem \ref{thm1} we may write $T(u)^{-1}=T(u^{-1})$ and $T(l)^{-1}=T(l^{-1})$.

Equation \eqref{eq:inv} provides a way to compute $A^{-1}$,
represented in the QT form, which consists essentially in computing
the coefficients of $u(z)$, $l(z)$ and of their inverses, and then to
invert a special QT matrix, that is, $I+T(a)^{-1}E_a=:I+E$.

Here we assume that the coefficients of the polynomials $u(z)$,
$l(z)$ and of the power series $u(z)^{-1}$ and $l(z)^{-1}$ are
available. In the appendix, we provide more details on how to perform
their computation.  Once we have computed $u(z)^{-1}$ and
$l(z)^{-1}$, by Theorem \ref{thm1} we may write
\begin{equation}\label{eq:inv1}
T(a)^{-1}=T(b)-H(l^{-1})H(u^{-1}),\quad b(z)=l(z^{-1})^{-1}u(z)^{-1},
\end{equation}
where the coefficients of $b(z)$ are computed by convolution of the
coefficients of $u(z)^{-1}$ and of $l(z^{-1})^{-1}$.

Concerning the inversion of $I+E$, where $E=T(a)^{-1}E_a$, we find
that $E=T(a)^{-1}U_aV_a^T=:UV^T$, for $U=T(a)^{-1}U_a$, $V=V_a$.
Consider the $k\times k$ matrix $S_k=I_k+V^TU$ which has finite
support since $V^TU=V_a^TT(a)^{-1}U_a$ and both $U_a$ and $V_a$ have a
finite number of nonzero rows. If $S_k$ is invertible then it can be
easily verified that $I-US_k^{-1}V^T$ is the inverse of $I+UV^T$, that
is, by the Shermann-Morrison formula, 
\begin{equation}\label{eq:inv2}
(I+UV^T)^{-1}=I-US_k^{-1}V^T, \quad S_k=I_k+V^TU.
\end{equation}

Now, combining \eqref{eq:inv}, \eqref{eq:inv1}, and \eqref{eq:inv2}
we may provide the following representation of the inverse:
\[
B:=A^{-1}=T(b)-H(l^{-1})H(u^{-1})-T(l^{-1})^T T(u^{-1})U_aS_k^{-1}V_a^T T(l^{-1})^T T(u^{-1}).
\]
Thus we may write $B$ in QT form as
\[
B:=T(b)+U_bV_b^T,\quad b(z)=l(z^{-1})^{-1}u(z)^{-1},
\]
where
\[
U_b=\left[H(l^{-1}), T(l^{-1})^T T(u^{-1})U_a S_k^{-1}
\right]\qquad\text{and}\qquad V_b=-\left[H(u^{-1}), T(u^{-1})^T
T(l^{-1}) V_a \right].
\]

In order to analyze the approximation errors in computing $A^{-1}$ as
a finitely representable  QT matrix, we assume that the computed values of
the Wiener-Hopf factors $u(z)$ and $l(z)$ are affected by some error
and that also in the process of computing the inverse of a power
series we introduce again some error.  Therefore, we denote by
$\widehat u(z)$ and $\widehat l(z)$ the computed values obtained in
place of $u(z)$ and $l(z)$, respectively in the Wiener-Hopf
factorization of $a(z)$ and set $e_{ul}(z)=a(z)-\widehat u(z)\widehat
l(z^{-1})$ for the residual error.  Moreover, denote by
$\delta_u(z)=u(z)-\widehat u(z)$, $\delta_l(z)=l(z)-\widehat l(z)$ the
absolute errors so that we may write the residual error as
\[
e_{ul}=l\delta_u+\widehat u\delta_l\doteq \widehat l\delta_u+\widehat u\delta_l.
\]

We indicate with $v(z)$ and $w(z)$ the power series reciprocal of
$\widehat u(z)$ and $\widehat l(z)$, respectively, i.e., such that
$\widehat u(z)v(z)=1$ and $\widehat l(z) w(z)=1$, while we denote with
$\widehat v(z)$ and $\widehat w(z)$ the polynomials obtained by
truncating $v(z)$ and $w(z)$ to a finite degree. Set
$e_{u}(z)=\widehat v(z)\widehat u(z)-1$, $e_l(z)=\widehat w(z)\widehat
l(z)-1$ for the corresponding residual errors.  We approximate
$a(z)^{-1}$ with the Laurent polynomial $\widehat b(z)=\widehat
w(z^{-1})\widehat v(z)$ up to the error $e_{inv}=a(z)\widehat b(z)-1$.
Finally, we write $\doteq$ and $\dotle$ if the equality and the
inequality, respectively, are valid up to quadratic terms in the
errors $e_{ul}(z)$, $e_u(z)$, and $e_l(z)$.  This way, we may
approximate the matrix $B=T(a)^{-1}=T(a^{-1})-H(l^{-1})H(u^{-1})$ with
the matrix $\widehat B=T(\widehat b)-H(\widehat w)H(\widehat v)$.

It is not complicated to relate  $B-\widehat B$ to the errors
$e_{inv}(z)$, $e_{ul}(z)$, $e_u(z)$, $e_l(z)$ as expressed in the
following proposition where, for the sake of notational simplicity, we
omit the variable $z$.

\begin{proposition}\label{prop:1} 
	The error $\mathcal {E} =T(a)^{-1}-\widehat B$, where $\widehat B=T(\widehat
	b)-H(\widehat w)H(\widehat v)$, is such that $\mathcal E
	=-T(a^{-1}e_{inv})+E_e$, $E_e=H(l^{-1}-\widehat
	w)H(u^{-1})+H(\widehat w)H(u^{-1}-\widehat v)$, and
	\[
	\|T(a^{-1}e_{inv})\|_2\le\|a^{-1}\|\w\|e_{inv}\|\w,\quad \|E_e\|_2\le \|l^{-1}-\widehat w\|\w\|u^{-1}\|\w+\|u^{-1}-\widehat v\|\w\|\widehat w\|\w.
	\]
	Moreover,
	for the errors $e_{inv}$, $e_u$, $e_l$ and $e_{ul}$ defined above
	it holds that
	\begin{equation}\label{eq:prop1.1}
	e_{inv}\doteq e_u+e_l+a^{-1}e_{ul}
	\doteq e_u+e_l+u^{-1}\delta_u+l^{-1}\delta_l.
	\end{equation}
	For the errors $l^{-1}-\widehat w$ and $u^{-1}-\widehat v$ it holds that
	\begin{equation}\label{eq:prop1.2}
	\begin{split}
	&l^{-1}-\widehat w= (l^{-1}-\widehat l^{-1})+(w-\widehat w)=
	-\widehat l^{-1}[l^{-1}\delta_l+e_l]
	\\ 
	&u^{-1}-\widehat v= (u^{-1}-\widehat u^{-1})+(v-\widehat v)=
	-\widehat u^{-1}[ u^{-1}\delta_u+e_u].
	\end{split}
	\end{equation}
	
\end{proposition}

\begin{proof}
	By linearity, we have $\mathcal E=T(a-\widehat
	b)-H(l^{-1})H(u^{-1})+H(\widehat w)H(\widehat
	v)=-T(a^{-1}e_{inv})+E_e$, where $E_e=H(l^{-1}-\widehat
	w)+H(\widehat w)H(u^{-1}-\widehat v)$, which, together with
	Theorem \ref{thm1}, proves the first part of the
	proposition. Observe that $a^{-1}=(e_{ul}+\widehat
	u\widehat l)^{-1}\doteq (\widehat u\widehat
	l)^{-1}(1-(\widehat u\widehat l)^{-1}e_{ul})$ so that, since
	$\widehat u^{-1}=v$ and $\widehat l^{-1}=w$, we may write
	\[
	\widehat b-a^{-1}=\widehat w\widehat v - wv+(\widehat u\widehat l)^{-2}e_{ul}=
	(\widehat w-w)\widehat v+w(\widehat v-v)(\widehat u\widehat l)^{-2}+(\widehat u\widehat l)^{-2}e_{ul}.
	\]
	Thus, since $\widehat w-w=\widehat l^{-1}e_l\doteq l^{-1} e_l$, and
	$\widehat v-v=\widehat u^{-1}e_u\doteq u^{-1}e_u$
	we arrive at
	\[
	\widehat b-a^{-1}\doteq a(e_l+e_u+ae_{ul}),
	\] 
	which proves \eqref{eq:prop1.1}. Equations \eqref{eq:prop1.2} are an
	immediate consequence of the definitions of $e_l$ and $e_u$.  
\end{proof}

Proposition \ref{prop:1} enables one to provide an upper bound to
$\|T(a)^{-1}-\widehat B\|$ in terms of $e_l$, $e_u$, $\delta_l$ and
$\delta_u$ as shown in the following corollary.

\begin{corollary}\label{coro1}
	For the error $\mathcal E=T(a)^{-1}-\widehat B$ it holds
	\[
	\|\mathcal E\|\qt 
	\dotle (\alpha \|a^{-1}\|+\|u^{-1}\|\w\|l^{-1}\|\w)\w(\|e_u\|\w+\|e_l\|\w+\|u^{-1}\|\w\cdot\|\delta_u\|\w
	+\|l^{-1}\|\w\cdot\|\delta_l\|\w).
	\]
\end{corollary} 

A similar analysis can be performed for the errors in the computed
inverse of $A=T(a)+E_a$. We omit the details.

We are ready to introduce a procedure to approximate the
inverse of $T(a)$, which is reported in Algorithm~1. 
If the UL factorization cannot be computed, then the routine returns
an error. The thresholds in the computation are adjusted to ensure
that the final error is bounded by $\epsilon$. The symbol $b(z)$ of 
$T(a)^{-1}$ is returned, along with $\hat v(z)$ and $\hat w(z)$ such that 
$T(a)^{-1}=T(\widehat b)+H(\widehat v)H(\widehat w)+\mathcal E$, where
$\mathcal E=-T(a^{-1}e_{inv})+E_e$, $E_e=H(\ell^{-1}-\widehat
w)H(u^{-1})+H(\widehat w)H(u^{-1}-\widehat v)$, and
$\|\mathcal E\|\dotle (\alpha\|a^{-1}\|\w+\|u^{-1}\|\w\|l^{-1}\|\w)\epsilon$.

\begin{algorithm}\label{alg:inverseToeplitz}
	
\algblock[TryCatchFinally]{try}{endtry}
\algcblock[TryCatchFinally]{TryCatchFinally}{finally}{endtry}
\algcblockdefx[TryCatchFinally]{TryCatchFinally}{catch}{endtry}
[1]{\textbf{catch} #1}
{\textbf{end try}}	
	
  \begin{algorithmic}[1]  
    \Procedure{InvertToeplitz}{$a(z)$, $\epsilon$}
      \try
        \State $[\hat u(z), \hat l(z)] \gets \Call{WienerHopf}{a(z), \frac{\epsilon}{4}}$
      \catch
        \State \Call{error}{``Could not compute UL factorization''} 
      \endtry
      \State $\hat v(z) \gets \Call{InversePowerSeries}{\hat u(z), 
      	  \epsilon / \norm{u^{-1}}_{\mathcal W}}$
    	\State $\hat w(z) \gets \Call{InversePowerSeries}{\hat l(z), 
      		\epsilon / \norm{l^{-1}}_{\mathcal W}}$
      \State $b(z) \gets \hat v(z) \hat w(z^{-1})$
      \State \Return $b(z), \hat v(z), \hat w(z)$
    \EndProcedure
  \end{algorithmic}
  
  \caption{Invert a semi-infinite Toeplitz matrix
  	with symbol $a(z)$ --- up to a certain error $\epsilon$.}
\end{algorithm}

From Corollary \ref{coro1} we find that 
$
\|\mathcal E\|\qt\le (\alpha\|a^{-1}\|\w+ \|l^{-1}\|\w\|u^{-1}\|\w)\epsilon
$.
Thus,
the correctness of the algorithm relies on Corollary \ref{coro1}
and on the existence of black boxes, which we will describe in the appendix,
that implement the functions \textsc{WienerHopf}$(\cdot)$ and
\textsc{InversePowerSeries}$(\cdot)$.
Relying on \eqref{eq:inv}, a similar algorithm and analysis can be given for the computation of $(T(a)+E_a)^{-1}$.

\subsection{Truncation and compression}\label{sec:compress}

We now describe in detail the implementation
of the operator $\mathcal{QT}$ on a finitely
generated QT matrix. 
The truncation of a QT matrix $A = T(a) + E_a$ 
is performed as follows: 
\begin{description}
    \item{(i)} Compute $\norm{A}_{\mathcal {QT}}$.
    \item{(ii)} Obtain a truncated version $\hat a(z)$ of 
      the symbol $a(z)$ by discarding the
      tails of the Laurent series. This has to be done ensuring that 
      $\norm{a - \hat a}_{\mathcal W} \leq \norm{A}_{\mathcal {QT}} \cdot \frac{\epsilon}{2\alpha}$. 
    \item{(iii)} Compute a compressed version $\hat E_a$ 
      of the
      correction using the SVD and  dropping negligible 
      rows and columns. Allow a truncation error
      bounded by $\norm{A}_{\mathcal {QT}}\cdot \frac{\epsilon}{2}$. 
\end{description}

The above choices of thresholds provide an approximation $\hat A$ 
to $A$ such that $\norm{A - \hat A}_{QT} \leq \norm{A}_{QT} \cdot \epsilon$. Notice that, the use of the QT-norm in the steps (ii) and (iii) enables to recognize unbalanced representations and to completely drop the negligible part.

When performing step (i), the only nontrivial step is to evaluate
$\norm{E_a}_2$. To this end, we compute an economy size SVD factorization of 
$E_a = U_a V_a^T$. This will 
also be useful  in step (iii) to perform the low-rank compression. 

In particular, we compute the QR factorizations $U_a= Q_U R_U$, 
$V_a= Q_V R_V$, so that,
$U_a V_a^T = Q_U R_U R_V^T Q_V^T$. Then, we compute the SVD of the matrix in the
middle $R_U R_V^T = U_R \Sigma V_R^T$. We thus obtain an SVD of the form 
$E_a = U \Sigma V^T$, where $U = Q_U U_R$, and $V = Q_V V_R$. This is computed
with $\mathcal O(nk^2)$ flops, where $n$ is the dominant dimension of the
correction's support. The value of $\norm{E_a}_2$ is obtained reading
off the largest singular value, i.e., the $(1,1)$ entry of $\Sigma$. 

 \begin{algorithm}
    \begin{algorithmic}[1]
       \Procedure{TruncateSymbol}{$a(z)$, $\epsilon$} 
         \If{\Call{Min}{$|a_{n_+}|, |a_{n_-}|$} $< \epsilon$}
           \If{$|a_{n_+}| < |a_{n_-}|$}
             \State $a(z) = a(z) - a_{n_+} z^{n_+}$
             \State $\epsilon \gets \epsilon - |a_{n_+}|$
           \Else
             \State $a(z) = a(z) - a_{n_-} z^{n_-}$
             \State $\epsilon \gets \epsilon - |a_{n_-}|$
           \EndIf
           \State $a(z) \gets \Call{TruncateSymbol}{a(z), \epsilon}$. 
         \EndIf
         \State \Return{$a(z)$}
       \EndProcedure
    \end{algorithmic}
    \caption{Truncate the symbol $a(z) = \sum_{n_- \leq j \leq n_+} a_j z^j$ to a given threshold 
    	$\epsilon$.}
    \label{alg:truncatesymbol}
 \end{algorithm}

Concerning step (ii), we repeatedly discard the smallest of the extremal coefficients
of $a(z)$, until the truncation errors do not exceed the specified threshold. 
In particular, we rely on Algorithm~\ref{alg:truncatesymbol} using 
 $\frac{\epsilon}{2 \alpha} \norm{A}_{\mathcal {QT}}$ as second argument. 

In step (iii), we first truncate the rank of $E_{a}$ by dropping singular
values smaller than $\frac{\epsilon}{4} \cdot \norm{A}_{\mathcal {QT}}$. To perform this step, 
we reuse the economy SVD computed at step (i). Then, we adopt a strategy similar
to the one of Algorithm~\ref{alg:truncatesymbol} to drop the last rows of 
$U$ and $V$. We set an initial threshold $\hat \epsilon = \frac{\epsilon}{4}\norm{A}_{\mathcal {QT}}$, 
and we drop either the last row  $U_n$ of $U$ or $V_m$ of $V$ 
if the norm of $U_n \Sigma$ (resp. $V_m \Sigma$) is smaller than the 
selected threshold. We then update 
$\hat{\epsilon} := \hat{\epsilon} - \norm{U_n \Sigma}$ (similarly
for $V_m \Sigma$) and repeat the procedure until $\hat \epsilon > 0$. This leads to a slightly pessimistic estimate, but ensures that the total truncation is within the desired bound. 
\subsubsection{Hankel compression}
When computing the multiplication of two Toeplitz matrices $T(a)$ and $T(b)$
it is necessary
to store a low-rank approximation of $H(a^-)H(b^+)$ (see Theorem \ref{thm1}).
In fact, storing  $H(a^-)$ and $H(b^+)$ directly can be
expensive whenever the Hankel matrices have large sizes e.g., when we
multiply two QT-matrices having wide Toeplitz bandwidths. However, the numerical rank
of the correction is often observed to be much lower than the dominant size of the correction's support. In such cases we  exploit the Hankel structure to cheaply obtain a compressed representation $E_c = U_c V_c^T$. We call this task  {\em Hankel compression}.

We propose two similar strategies for addressing Hankel compression. 
The first is to rely on a Lanczos-type method, in the form of the
Golub-Kahan bidiagonalization procedure \cite{paige1974bidiagonalization}. This
can be implemented by using matrix-vector products of the form 
$y = Ax$ or $y = A^{T}x$, where $A = H(a^{-}) H(b^+)$ is a product
of two Hankel matrices. The product $y = Ax$ can be computed 
in $\mathcal O(n \log n)$ time using the FFT. Since the Hankel
matrices are symmetric, the multiplication by $A^T$ is obtained 
swapping the role of $a^-$ and $b^+$. 

This approach has an important advantage: the rank can be determined
adaptively while the Lanczos process builds the basis, and assuming 
that Lanczos converges in $\mathcal O(k)$ steps, with $k$ being the
numerical rank of $A$, then the complexity is $\mathcal O(k n \log n)$ flops: 
much lower than a full QR factorization or SVD. 

A second (similar) strategy is to use random sampling techniques 
\cite{halko2011finding}, which rely on the evaluation of the 
product $AU$, with $U$ being a matrix with normally distributed
entries. If the columns of $AU$ span the range of $A$, then
we extract an orthonormal basis of it, and we use it to cheaply compute
the SVD of $A$ \cite{halko2011finding}. In the implementation the number 
of columns of $U$ is chosen adaptively, enlarging it until a sufficient
accuracy is reached. 
The product $AU$ can be efficiently computed using
the FFT, and it's possible to obtain BLAS3 speeds by re-blocking.

Both strategies are implemented in the package, and the user can
select the Lanczos-type algorithm running
\lstinline|cqtoption('compression', 'lanczos')|
or  the one based on random sampling with the command: 
\lstinline|cqtoption('compression', 'random')|.

The performance of the two approaches is very similar. In Figure~\ref{fig:hankel-comp}
the timings for the compression of the product of two $n \times n$ Hankel
matrices are reported. The
symbol has been chosen drawing from a uniform distribution 
enforcing an exponential decay as 
follows: 
\begin{equation} \label{eq:random-hankel}
  a(z) = \sum_{j \in \mathbb Z^+} a_j z^j, \qquad 
  a_j \sim \lambda(0, e^{- \frac{j}{10}}), 
\end{equation}
where $\lambda(a, b)$ is the uniform distribution on $[a, b]$. In the
example reported in Figure~\ref{fig:hankel-comp}, the numerical rank (up
to machine precision) of the product of the Hankel matrices generated
according to \eqref{eq:random-hankel}
is $90$. The
break-even point with a full SVD is around $500$ in this example, 
and this behavior is typical. Therefore, we use a dense singular value 
decomposition for small matrices ($n \leq 500$), and we resort to
Lanczos or random sampling (depending on user's preferences) otherwise. 

In the right part of Figure~\ref{fig:hankel-comp} we report
also the accuracies by taking the relative residual
$\norm{UV^T - H(a^-) H(b^+)}_2 / \norm{H(a^-) H(b^+)}_2$. Since
the norms are computed as dense matrices, we only test this up 
to $n = 4096$. The truncation threshold in this example is set 
to $10^{-14}$. 

\begin{figure}
	\centering
	\begin{minipage}{.55\linewidth}
			\begin{tikzpicture}
			\begin{loglogaxis}[width=.95\linewidth, legend pos = north west, 
			height = .3\textheight, xlabel = $n$, ylabel = CPU time (s) ,
			xmin = 25]
			\addplot table[x index = 0, y index = 1] {hankel-comp.dat};
			\addplot table[x index = 0, y index = 2] {hankel-comp.dat};
			\addplot table[x index = 0, y index = 3] {hankel-comp.dat};
			\addplot[domain = 128:65536,dashed, thin] {2e-6 * x * ln(x)};
			\legend{Lanczos, Random, SVD, $\mathcal O(n \log n)$};
			\end{loglogaxis}
			\end{tikzpicture}
	\end{minipage}~~~~~\begin{minipage}{.4\linewidth}
	  \centering \textbf{Accuracy} \\[6pt]
	  \pgfplotstabletypeset[
	  columns={0,4,5},
	  columns/0/.style={column name = Size},
	  columns/4/.style={column name = Lanczos},
	  columns/5/.style={column name = Random},
	  skip rows between index={6}{12}
	  ]{hankel-comp.dat}
\end{minipage}

	\caption{On the left, 
		timings required to compress a product of two $n \times n$ Hankel
		matrices with decaying coefficients, for different values of $n$, and using
		different strategies. The tests with the dense singular value decomposition
		have been run only up to $n = 4096$. The other methods have been tested up
		to $n = 2^{16}$. On the right, the accuracies, up to size $4096$, in the $2$-norm
		achieved by setting
		the truncation threshold to $10^{-14}$.}
	\label{fig:hankel-comp}
\end{figure}
\subsection{Finite quasi Toeplitz matrices}
\label{sec:finite}
The representation and the arithmetic, introduced so far, are here adapted for handling finite size matrices of the form Toeplitz plus correction. Clearly, all the matrices of finite size can be represented in this way. This approach is convenient only if the corrections of the matrices involved in our computations are either sparse or low-rank. 
Typically, this happens when the input data of the computation are banded Toeplitz matrices.

In what follows, given a Laurent series $a(z)$ we indicate with $T_{n,m}(a)$ the $n\times m$ Toeplitz matrix obtained selecting the first $n$ rows and $m$ columns from $T(a)$. 
Given a power series $f(z) = \sum_{j \geq 0} f_j z^j$ we denote by 
$H_{n,m}(f)$ the $n\times m$ Hankel matrix whose non-zero
 anti-diagonal elements correspond to $f_1,f_2,\dots,f_{\min\{n,m\}}$. Finally, given a Laurent polynomial $a(z)=\sum_{j=-n+1}^{m-1}a_j z^j$ we indicate with $\widetilde a(z)$ the shifted Laurent polynomial $z^{n-m}a(z)$.

In order to extend the approach, it is essential to look at the following variant of Theorem~\ref{thm1}, for finite Toeplitz matrices \cite{widom}.
\begin{theorem}[Widom]\label{thm:widom}
	Let $a(z)=\sum_{-n+1}^{m-1}a_jz^j$, $b(z)=\sum_{-m+1}^{p-1}b_jz^j$ and set $c(z)=a(z)b(z)$. Then
	\[
	T_{n,m}(a)T_{m,p}(b)=T(c)-H_{n,m}(a^-)H_{m,p}(b^+)-J_nH_{n,m}(\widetilde a^+)H_{m,p}(\widetilde b^-)J_p
	\]
	where $J_n=\begin{bmatrix}
	&&1\\ &\iddots\\
	1
	\end{bmatrix}\in\mathbb R^{n\times n}$ is the flip matrix.
\end{theorem}
An immediate consequence of Theorem~\ref{thm:widom} is the following extension of the Wiener-Hopf factorization for square Toeplitz matrices.
\begin{corollary}
Let $a(z)=\sum_{-n+1}^{n-1}a_jz^j$  and let $a(z)=u(z)l(z^{-1})$
be its Wiener-Hopf factorization. Then	
\[
T_{n,n}(a)= T_{n,n}(u)T_{n,n}(l)^T + J_nH_{n,n}(u)H_{n,n}(l)J_n.
\]
\end{corollary}

The above results introduce an additional term with respect to their counterparts for semi-infinite matrices. In particular, if the lengths of the symbols involved is small, compared to the dimension of the matrices, then the support of the non-Toeplitz component is split in two parts located in the upper left corner and in the lower right corner, respectively. This suggests to consider  two separate corrections. 

Handling two separate corrections is convenient as long as they do not overlap. When
this is the case, we represent
finite quasi-Toeplitz matrices by storing two additional matrices, that represent the lower right correction in factorized form. More precisely, $A\in\mathbb R^{n\times m}$ is represented with two vectors, storing the symbol, and with the matrices $U_a,V_a,W_a,Z_a$ such that $U_aV_a^T$ and $J_nW_aZ_a^TJ_m$ correspond to the corrections in the corners. As a practical example we report two possible options for defining the matrix
\[
A= \begin{bmatrix}
1&3\\
-2&\ddots&\ddots\\
&\ddots&\ddots&\ddots\\
&&\ddots&\ddots&3\\
&&&-2&1\\
\end{bmatrix}
+
\begin{bmatrix}
1&1\\
1&1\\
&\\
&\\
&&&&1&2&3\\
&&&&2&4&6
\end{bmatrix}\in\mathbb R^{12\times 12}.
\]
\begin{lstlisting}
>> n = 12;
>> E = ones(2,2);
>> F = [1 2 3; 2 4 6];
>> pos = [1 3];
>> neg = [1 -2];
>> A = cqt(neg, pos, E, F, n, n);
\end{lstlisting}
Once again, we also give the possibility to directly specify the corrections in the factorized form.
\begin{lstlisting}
>> U = [1; 1];
>> V = U;
>> W = [1; 2];
>> Z = [1; 2; 3];
>> A = cqt(neg, pos, U, V, W, Z, n, n);
\end{lstlisting}

The arithmetic operations described for semi-infinite QT matrices can be analogously
extended to this case. In the next section we describe in more detail how to
handle the corrections when performing these operations. 

In particular, when the corrections overlap, we switch to a single correction format, 
as in the semi-infinite case, where the support of the correction corresponds
to the dimension of the matrix. In practice
this is done by storing it as an upper-left correction, setting the lower right
to the empty matrix. 
	For this approach to be convenient, the rank of the correction needs to stay
	small compared to the size of the matrix. In fact, only the sparsity is lost, but the
	data-sparsity of the format is still exploitable. 

\subsubsection{Handling the corrections in the computations}

Let $A = T(a) + U_a V_a^T + J_n W_a Z_a^T J_n$ and 
$B = T(b) + U_b V_b^T + J_n W_b Z_b^T J_n$. 
For simplicity, we assume that $A$ and $B$ 
are square, of dimension $n \times n$. 
Analogous statements hold in the rectangular case, which has been implemented
in the toolbox. 

As we already pointed out, we need to check that, while manipulating finite
QT-matrices, the corrections do not overlap. 
More precisely, if the top-left correction of $A$ is of dimension
$i_u^{(A)} \times j_u^{(A)}$ and the bottom
one is $i_l^{(A)} \times j_l^{(A)}$, then we ask that at least one between $i_u^{(A)} + i_l^{(A)}$ and
$j_u^{(A)} + j_l^{(A)}$ is smaller than $n$ (and analogous conditions for $B$). 
The possible configurations of the two corrections
are reported in Figure~\ref{fig:sum}. This property might not be
preserved in the outcome of arithmetic operations. 

Therefore, we need to understand how the supports
of the corrections behave under sum and multiplications. Concerning the
sum, the support of the correction in $A + B$
is contained in the union of the two supports in $A$ and $B$. The multiplication
is slightly more involved. First, we check if the products $V_a^T J_n W_b$ and 
$Z_a^T J_n U_b$ are both zero, that is they do not ``interact''. This happens when 
$j_u^{(A)} + i_l^{(B)}$ and 
$j_l^{(A)} + i_u^{(B)}$ are both smaller than $n$. Second, we need to consider all the
addends contributing to the correction; to this end, we note that the product of the
top corrections has support $i_u^{(A)} \times j_u^{(B)}$ and the product of the bottom
ones $i_l^{(A)} \times j_l^{(B)}$. Moreover, the multiplication of the Toeplitz
part by the correction enlarges the support of the latter by the bandwidth of the former. 
Finally, the Hankel products have a support that depends on the length of the symbols. 
If the first condition is met and all these contributions satisfy
the non-overlapping property for the sum, we keep the format with two separate
corrections. Otherwise, we merge them into a single one. 
An analogous analysis is valid for the inversion, since it is implemented by a combination of
sum and multiplications by means of the Woodbury formula. 

\begin{figure}
	\centering
	\begin{tikzpicture}
	  \node at (1,-.5) {a)};
      \filldraw[mygray2] (0,2) rectangle (.4,1.5);
      \filldraw[mygray2] (2,0) rectangle (1.3,.9);
      \draw (0,0) rectangle (2,2);
      
      \node at (4,-.5) {b)};
      \filldraw[mygray2] (3,2) rectangle (4.1,1.5);
      \filldraw[mygray2] (5,0) rectangle (3.7,.9);
      \draw (3,0) rectangle (5,2);
      
      \node at (7,-.5) {c)};
      \filldraw[mygray2] (6,2) rectangle (6.4,0.8);
      \filldraw[mygray2] (8,0) rectangle (7.3,1.2);
      \draw (6,0) rectangle (8,2);

      \node at (10,-.5) {d)};
      \filldraw[mygray2] (9,2) rectangle (10.2,.8);
      \filldraw[mygray2] (11,0) rectangle (10,1.2);            
      \draw (9,0) rectangle (11,2);
	\end{tikzpicture}
	
	\caption{Possible shapes of the corrections in the the representation
		of a finite QT-matrix. The
		corrections are disjoint in cases a) -- c), but not in case d). In this
		last case it is convenient to store the correction entirely as a top
		correction.}
    \label{fig:sum}
\end{figure}
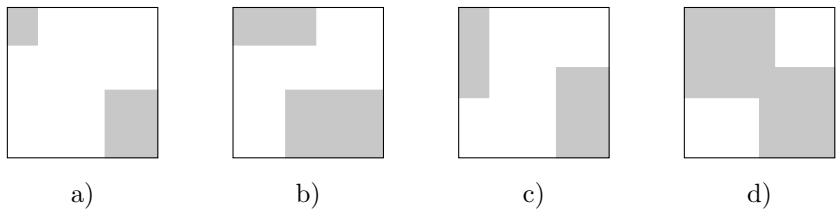

\subsection{Other basic operations}

By exploiting the structure of the
data, it is possible to devise efficient implementations of common operations, 
such as the computation of matrix functions, norms, and extraction of submatrices. 

The functions reported in Table~\ref{tab:functions} have been implemented in the
toolbox relying on the QT arithmetic. For a detailed description of the various
functions, the user can invoke the \texttt{help} command. For instance, the 
matrix exponential is implemented using the Pad\'e formula combined with a scaling
and squaring strategy. An implementation based on the Taylor expansion is
also available by calling \texttt{expm(A, 'taylor')}. 

In particular, the extraction of finite submatrices can be useful to inspect
parts of infinite QT matrices, and also to obtain finite sections. 

\begin{table}
    \centering
    \begin{tabular}{cc}
      Function & Description \\ \hline 
      \texttt{A(I,J)} & Extract submatrices of a QT-matrix, for integer vectors $I$ and $J$. \\
      \texttt{A\textasciicircum p} & Integer power of a QT matrix. \\
      \texttt{cr}    & Cyclic reduction for quadratic matrix equations. \\
      \texttt{expm}  & Computation of the matrix exponential $e^A$.  \\
      \texttt{funm}  & Computations of matrix functions by contour integration.\\
      \texttt{norm}  & Computation of the QT norm and, only in the finite case, 
         of the $p$-norms for $p = 1, 2, \infty$. \\
      \texttt{sqrtm} & Computation of the matrix square root (Denman--Beavers iteration). \\
      \texttt{ul} & Compute the $UL$ factorization of any QT-matrix. \\
    \end{tabular}
    \caption{List of implemented functions in \toolbox.}
    \label{tab:functions}
\end{table}

\begin{remark} \label{rem:inversion}
     All the arithmetic operations, with the only exception of the inversion, can be
     performed in $\mathcal O(n \log n)$ time relying on the FFT. 
     The current implementation of the Wiener-Hopf factorization (see
     the Appendix), 
     required for the inversion,
     needs $\mathcal O(b^3)$ where $b$ is the numerical bandwidth of the Toeplitz part --- but this complexity can be lowered. For instance, one
     can obtain a sub-quadratic complexity combining the FFT-based polynomial 
     inversion with a superfast Toeplitz solver for the computation of the Wiener-Hopf factorization. 
     This is not currently implemented in the toolbox, and will be investigated in future
     work. 
\end{remark}

\section{Examples of applications}
\label{sec:app}

In this section we show some applications and examples of computations
with $\mathcal{QT}$ matrices. Here we focus on the computation of 
matrix functions and solving matrix equations. Other examples related to matrix equations
have already been shown in \cite{bini2016semi,biniquadratic}. 

In all our experiments we have set the truncation tolerance 
to $\epsilon := 10^{-15}$. The algorithm used for compression is the Lanczos method. 
Accuracy and timings are analogous if the random sampling-based compression
is used. 

The tests have been performed on a laptop with an i7-7500U CPU running at 2.70GHz 
with 16GB of memory, using \matlab\ R2017a.

\subsection{Exponential of Toeplitz matrices}

Recently there has been a growing interest in the computation of functions of 
Toeplitz matrices. For instance, in \cite{lee2010shift} the authors consider
the problem of option pricing using the Merton model. This requires computing
the exponential of a dense non-symmetric Toeplitz matrix. A fast method for this
problem has been developed in \cite{kressner2016fast}. 

We refer to \cite[Example 3]{lee2010shift} for the details on the model; the Toeplitz
matrix obtained has symbol $a(z)$ with coefficients 
\[
  a_j = \begin{cases}
    \displaystyle
    \phi(0) + 2b - r - \lambda & j = 0 \\
    \phi(j\Delta_\xi) + b + j c & j = \pm 1 \\
    \phi(j \Delta_\xi) & \text{otherwise} \\
  \end{cases}, \qquad 
  b = \frac{\nu^2}{2 \Delta_\xi^2}, \qquad 
  c = \frac{2r - 2\lambda \kappa - \nu^2}{4\Delta_\xi}
\]
where 
\[
  \phi(\eta) := \lambda \Delta_\xi \frac{e^{-(\eta - \mu)^2 / (2 \sigma^2)}}{\sqrt{2\pi} \sigma}. 
\]
The values of the parameters are chosen as in \cite{kressner2016fast}; for the 
Toeplitz matrix of size $n \times n$ we have: 
\[
  r = 0.05,\ \  \lambda = 0.01,\ \ \mu = -0.9,\ \ \nu = 0.25,\ \ \sigma = 0.45,\ \ \kappa = e^{\frac{\mu+\sigma^2}{2}} - 1,\ \ \Delta_\xi = \frac{4}{n + 1}. 
\]

In Figure~\ref{fig:kressner} we report the timings for the computation of 
the matrix exponential $e^{T_n}$ for different values of $n$. The
CPU time is compared with the function $O(n\log n)$.
 The 
accuracy for the cases where $n \leq 4096$, where we could compare
with $\texttt{expm}$, are reported in Figure~\ref{fig:kressneraccuracy}. 

In particular, we report the relative error in the Frobenius norm
$\norm{\texttt{expm}(A) - E}_F / \norm{\texttt{expm}(A)}_F$,
 where
$E$ is the approximation of $e^{A}$ computed by the toolbox
using the Taylor approximant. We
compare it with the quantity $\norm{A}_F \cdot \epsilon$, which is 
a lower bound for the condition number of the matrix exponential 
times the truncation threshold used in the computations (see \cite{higham2008functions}). 
From Figure \ref{fig:kressneraccuracy} one can see that the errors are bounded by $\norm{A}_F \cdot \epsilon$
and $10 \cdot \norm{A}_F \cdot \epsilon$.

We have used a scaling
and squaring scheme, combined with a Taylor approximant of order $12$, 
to compute the matrix exponential. In this case, where the bandwidth of the
Toeplitz part is non-negligible, this approach is more efficient than a 
Pad\'e approximant that requires an inversion (see Remark~\ref{rem:inversion}).

\begin{figure}
	\begin{minipage}{.55\linewidth}
	  \begin{tikzpicture}
	  \begin{loglogaxis}[legend pos = north west]
	  \addplot table[x index = 0, y index = 1] {kress_luce.dat};
	  \addplot[domain = 256:131000, dashed] {5e-4 *x*ln(x)};
	  \legend{Time (s), $\mathcal O(n \log n)$}
	  \end{loglogaxis}
	  \end{tikzpicture}
	\end{minipage}~\begin{minipage}{.3\linewidth}
    	    \pgfplotstabletypeset[
    	      columns = {0,1,4},
    	      columns/0/.style = {column name = Size},
    	      columns/1/.style = {column name = Time (s)},
    	      columns/4/.style = {column name = Corr. rank},
    	    ]{kress_luce.dat}
    \end{minipage}
    \caption{Timings for the computation of the matrix exponential on the Merton model. 
    	The rank of the correction is reported in the last column of the table.}
    \label{fig:kressner}
\end{figure}

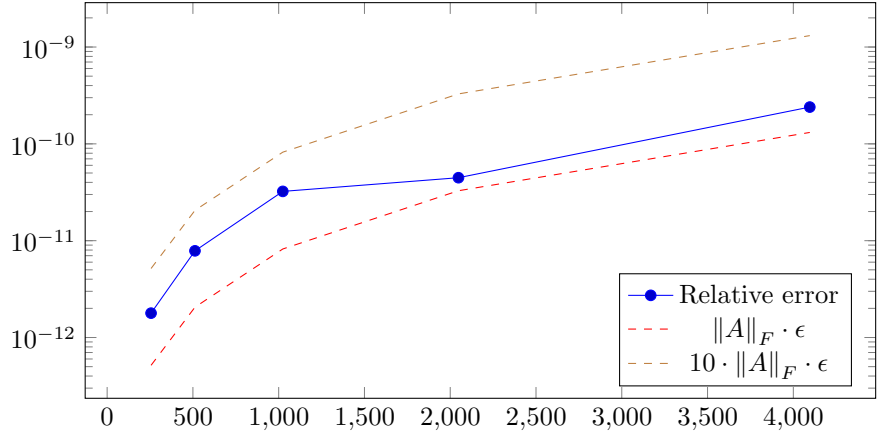
\begin{figure}
	\centering
	\begin{tikzpicture}
      \begin{semilogyaxis}[legend pos = south east, width = .8\linewidth, height = .35\textheight]
        \addplot table[x index = 0, y index = 2]{kress_luce.dat};
        \addplot[dashed, red] table[x index = 0, y expr = \thisrowno{3} * 1e-15]{kress_luce.dat};
        \addplot[dashed, brown] table[x index = 0, y expr = \thisrowno{3} * 1e-14]{kress_luce.dat};
        \legend{Relative error, $\norm{A}_F \cdot \epsilon$, $10 \cdot \norm{A}_F \cdot \epsilon$}
      \end{semilogyaxis}
	\end{tikzpicture}
	\caption{Relative error with the Frobenius norm of the computed 
		matrix exponential, compared with a lower bound for the condition
		number of the matrix exponential multiplied by the truncation
		threshold used in the computation ($\epsilon := 10^{-15}$).}
	\label{fig:kressneraccuracy}
\end{figure}

\subsection{Computing the square root of a semi-infinite matrix}

We show another application to the computation of the square root of an
infinite QT-matrix $A$. We consider the infinite matrix $A = T(a) + E_a$, where
\[
  a(z) = \frac{1}{4} \left( z^{-2} + z^{-1} + 1 + 2z + z^2 \right), 
\]
and $E_a$ is a rank $3$ correction in the top-left corner of norm
$\frac{1}{5}$ and support ranging from $32$ to $1024$ rows and columns. 
The square root can be computed using the Denman--Beavers iteration, 
which is implemented in the toolbox and accessible using 
\lstinline-B = sqrtm(A);-. We report the timings and the residual
$\norm{B^2 - A}_{\mathcal{QT}}$ of the
computed approximations in Table~\ref{tab:sqrtm}. Moreover, 
the rank and support of the correction in $A^{\frac 1 2}$ are
reported in the last three columns of the table. One can see that 
the rank stays bounded, and that the support does not
increase much. The CPU time takes negligible values even for large support of the correction.

\begin{table}
	\centering
    \pgfplotstabletypeset[
    columns/0/.style = {column name = Initial corr. support},
    columns/1/.style = {column name = Time (s)},
    columns/2/.style = {column name = Residual},
    columns/3/.style = {column name = Rank},
    columns/4/.style = {column name = Corr. rows},
    columns/5/.style = {column name = Corr. cols}
    ]{sqrtm.dat}
    \caption{Timings and residuals for the computations of the square
    	root of an infinite Toeplitz matrix with a square top-left correction
    	of variable support. The final rank and correction support in the
    	matrix $A^{\frac 1 2}$ are reported in the last 3 columns.}
    \label{tab:sqrtm}
\end{table}

\subsection{Solving quadratic matrix equations}

Finally, we consider an example arising from the analysis of a random
walk on the semi-infinite strip $\{ 0, \ldots, m \} \times \mathbb N$. We assume
the random walk to be a Markov chain, and that movements are possible
only to adjacent states, that is, from $(i,j)$ one can reach only $(i',j')$ with $|i-i'|,|j-j'|\leq 1$, with probabilities of moving up/down and left/right
not depending on the current state. Then the transition matrix $P$ is 
an infinite Quasi-Toeplitz-Block-Quasi-Toeplitz matrix of the form
\[
  P = \begin{bmatrix}
    \hat A_0 & A_1 \\
    A_{-1}   & A_0 & A_1 \\
    & \ddots & \ddots & \ddots \\
  \end{bmatrix}, 
\]
and the problem of computing the invariant vector $\pi$  
requires to solve the $m \times m$ quadratic matrix equation
$A_{-1} + A_0 G + A_1 G^2 = G$ \cite{bini2009cyclic}. The matrices $A_i$ 
are non negative tridiagonal Toeplitz matrices with corrections to the elements
in position $(1,1)$ and $(m,m)$, and satisfy
$(A_{-1} + A_0 + A_1)e = e$, where $e$ is the vector of all ones. 

The solution $G$ can be computed, for instance, using Cyclic reduction
(see the Appendix for the details) -- a matrix iteration involving
matrix products and inversions. We consider an example where
the transition probabilities  are chosen
in a way that gives the following symbols: 
\[
  a_{-1}(z) = \frac 14 (2z^{-1} + 2 + 2z), \qquad 
  a_{0}(z) = \frac{1}{10} (z^{-1} + 2z), \qquad 
  a_{1}(z) = \frac{1}{6} (3z^{-1} + 6 + 2z), 
\]
properly rescaled in order to make $A_{-1}+A_0+A_1$ a row-stochastic matrix. The top
and bottom correction are chosen to ensure stochasticity on the first
and last row.

\begin{figure}
	\begin{minipage}{.5\linewidth}
		\bigskip
		
		\bigskip
		
		\begin{tikzpicture}
		\begin{loglogaxis}[legend pos = south east, xlabel = $m$, 
		ylabel = Time (s), width = .95\linewidth, ymin = 1e-3, 
		height = .4\textheight]
		\addplot table[x index = 0, y index = 1] {quadratic.dat};
		\addplot table[x index = 0, y index = 4] {quadratic.dat};
		\addplot table[x index = 0, y index = 5] {quadratic.dat};
		\legend{\toolbox, Dense, HODLR};
		\end{loglogaxis}
		\end{tikzpicture}
	\end{minipage}~~\begin{minipage}{.45\linewidth}
	\pgfplotstabletypeset[
	columns = {0,1,2,4,5},
	columns/0/.style = {column name = $m$}, 
	columns/1/.style = {column name = $t_{\mathrm{cqt}}$},	     
	columns/2/.style = {column name = Corr. rank},	     
	columns/4/.style = {column name = $t_{\mathrm{dense}}$},
	columns/5/.style = {column name = $t_{\mathrm{HODLR}}$}
	]{quadratic.dat}
\end{minipage}

\caption{On the left, timings for the solution of the quadratic equation  
	$A_{-1} + A_0 G + A_1 G^2 = G$ arising from the 
	random walk on an infinite strip. On the right the timings and 
	the ranks of the final correction are reported in the table.}
\label{tab:quadratic}
\end{figure}

We compare the performances of a dense iteration (without exploiting
any structure) -- with the same one implemented using 
\toolbox, and also with a fast $\mathcal O(m \log^2 m)$ method
which exploits the tridiagonal structure relying on the arithmetic of 
hierarchical matrices (HODLR) \cite{bini2017decay,bini2017efficient}.
In Figure~\ref{tab:quadratic}, one can see  that the
timings of the dense solver are lower for small dimensions -- but the
ones using the toolbox do not suffer from the increase in the 
dimension. The dense solver was tested only up to dimension
$m = \num{4096}$. 

The implementation relying on \toolbox\ is faster already for dimension 
$512$, and has the remarkable property that the time does not depend on
the dimension. This is to be expected, since the computations are all
done on the symbol (which is dimension independent), and on the corrections, 
which only affect top and bottom 
parts of the matrices. 

The residual of the quadratic matrix equation $\norm{A_{-1} + A_0 G + A_1 G^2}$ is 
bounded in the $\mathcal QT$ norm by approximately $7 \cdot 10^{-12}$ in
all the tests, independently of the dimension, when the \toolbox\
solver is used.

\section{Conclusions}\label{sec:conclusion}
We have analyzed the class of  Quasi Toeplitz matrices, introduced a suitable norm and a way to approximate any QT matrix by means of a finitely representable matrix within a given relative error bound. Within this class, we have introduced and analyzed, in all the computational aspects, a matrix arithmetic. We have provided an implementation of QT matrices and of their matrix arithmetic in the form of a Matlab toolbox. The software {\tt cqt-toolbox}, available at \url{https://github.com/numpi/cqt-toolbox}, has been tested with both semi-infinite QT matrices and with finite matrices represented as the sum of a Toeplitz matrix and a correction.   This software has shown to be very efficient in computing matrix functions and solving matrix equations encountered in different applications.
\appendix
\section{Appendix}
\label{sec:appendix}

Here we report the main algorithms that we have implemented to perform
inversion of QT matrices, namely, the Sieveking-Kung algorithm
\cite{bp:book} for inverting triangular Toepliz matrices (or power
series), and an algorithm based on Cyclic Reduction
\cite{bini2009cyclic} to compute the Wiener-Hopf factorization of a
symbol $a(z)$. We also provide a general view of the available
algorithms for the Wiener-Hopf factorization \cite{bibo}, \cite{bfgm},
\cite{boett},  with an outline of their relevant
computational properties. Choosing the more convenient algorithm for
this factorization depends on several aspects like the degree of
$a(z)$, and the location of its zeros, and this is an issue to be
better understood.

\subsection{The Sieveking-Kung algorithm}\label{sec:sk}
We shortly recall the Sieveking-Kung algorithm for computing the
first $k+1$ coefficients $v_0,\ldots,v_k$ of
$v(z)=\sum_{i=0}^{\infty}v_iz^i$ such that $v(z)u(z)=1$, or
equivalently, the first $k$ entries in the first row of
$T(u)^{-1}$. For more details we refer the reader to the book
\cite{bp:book}.

For notational simplicity, denote $V_q$ the $q\times q$ leading
submatrix of $T(u)$. Consider $V_{2q}$ and partition it into 4 square
blocks of size $q$:
\[
V_{2q}=\begin{bmatrix} V_q& S_q\\0&V_q
\end{bmatrix}
\]
so that
\[
V_{2q}^{-1}=\begin{bmatrix}
V_q^{-1}& -V_q^{-1}S_qV_q^{-1} \\ 0 & V_q^{-1}
\end{bmatrix}.
\]
Since the inverse of an upper triangular Toeplitz matrix is still upper
triangular and Toeplitz, it is sufficient to compute the first row
of $V_{2q}^{-1}$. The first half clearly coincides with the first
row of $V_q^{-1}$, the second half is given by
$
-e_1^T V_q^{-1}S_q V_q^{-1},
$ 
where $e_i$ is the vector with the $i$-th component equal to 1 and with
the remaining components equal to zero.

Thus the algorithm works this way: For a given (small) $q$ compute the
first $q$ components by solving the system $V_q^T x=e_1$. 
Then, by subsequent doubling steps, compute $2q$, $4q$,
$8q,\ldots$, components until some stop condition is
satisfied. Observe that, denoting $v_q(z)$ the polynomial obtained at
step $q$, the residual error $r_q(z)=a(z)v_q(z)-1$ can be easily
computed so that the stop condition $\|r_q\|\w\le\epsilon\|a\|\w$ can
be immediately implemented. Concerning the convergence speed, it must
be pointed out that the equation $r_{2q}(z)=r_q(z)^2$ holds true (see
\cite{bp:book}), implying that the convergence to zero of the norm of the
residual error is quadratic.

This approach has a low computational cost since the products Toeplitz
matrix by vector can be implemented by means of FFT for an overall
cost of the Sieveking-Kung algorithm of $O(n\log n)$ arithmetic
operations.

This algorithm, here described in matrix form, can be equivalently
rephrased in terms of polynomials and power series.

\subsection{The Wiener-Hopf factorization}\label{sec:wh}
We recall and synthesize the available algorithms for computing the
coefficients of the polynomials $u(z)$ and $l(z)$ such that
$a(z)=u(z)l(z^{-1})$ is the Wiener-Hopf factorization of $a(z)$.
Denote $\xi_i$ the zeros of $a(z)$ ordered so that
$|\xi_i|\le|\xi_{i+1}|$. This way, $|\xi_{n_+}|<1<|\xi_{1+n_+}|$,
moreover, $u(\xi_i)=0$ for $i=1,\ldots,n_+$ while $l(\xi_i^{-1})=0$
for $i=n_++1,\ldots,n_++n_-$.

A first approach is based on reducing the problem to solving a
quadratic matrix equation. Let $p\ge\max(n_-,n_+)$, reblock the
matrices in the equation $T(a)=T(u)T(l^{-1})$ into $p\times p$
blocks and obtain
\[
\begin{bmatrix}
A_0&A_1\\
A_{-1}&A_0&A_1\\
&\ddots&\ddots&\ddots
\end{bmatrix}
=\begin{bmatrix}
U_0&U_1\\&U_0&U_1\\
&&\ddots&\ddots
\end{bmatrix}
\begin{bmatrix}
L_0\\L_1&L_0\\
&L_1&L_0\\
&&\ddots&\ddots
\end{bmatrix}
\]
where, by using the \matlab\ notation, 
\[\begin{split}
&A_0=\hbox{toeplitz}([a_0,\ldots,a_{-p+1}],[a_0,\ldots,a_{p-1}]),\\
&A_1=\hbox{toeplitz}([a_p,a_{p-1},\ldots,a_1],[a_p,0,\ldots,0]),\\
&A_{-1}=\hbox{toeplitz}([a_{-p},0,\ldots,0],[a_{-p},\ldots,a_{-1}]),
\end{split}
\]
and $a_i=0$ if $i$ is out of range.

Set $W=U_0L_0$, $R=-U_1U_0^{-1}$, $G=-L_0^{-1}L_1$ and get the factorization
{\small
	\[
	\begin{bmatrix}
	A_0&A_1\\
	A_{-1}&A_0&A_1\\
	&\ddots&\ddots&\ddots
	\end{bmatrix}
	=\begin{bmatrix}
	I&-R\\&I&-R\\
	&&\ddots&\ddots
	\end{bmatrix}
	\begin{bmatrix}
	W\\ &W\\&&\ddots
	\end{bmatrix}
	\begin{bmatrix}
	I\\-G&I\\
	&-G&I\\
	&&\ddots&\ddots
	\end{bmatrix}.
	\]} 
Multiplying the above equation to the right by the block column vector
with entries $I,G,G^2,G^3,\ldots$ or multiplying to the left by the
block row vector with entries $I, R,R^2,R^3,\ldots$ one finds that the
matrices $R$ and $G$ are solutions of the equations
\begin{equation}\label{eq:mateq}
A_1G^2+A_0G+A_{-1}=0,\qquad R^2A_{-1}+RA_0+A_1=0
\end{equation}
and have eigenvalues $\xi_1,\ldots,\xi_{n_+}$ and $\xi_{n_+
  +1}^{-1},\ldots,\xi_{n_++n_-}^{-1}$, respectively, so that they have
spectral radius less than 1. For more details in this regard we refer
the reader to \cite{bgm:laa}.

Observe that, since
\[
G=-\begin{bmatrix}
l_0\\
l_1&l_0\\
\vdots&\ddots&\ddots\\
l_{p-1}&\ldots&l_1&l_0
\end{bmatrix}^{-1}\begin{bmatrix}
l_p&\ldots&l_1\\
&\ddots&\vdots\\
&      &l_p
\end{bmatrix}
\]
then $Ge_{p-n_-+1}=-l_{n_-} L_0^{-1}e_1$, while
$e_1^TG=-l_0^{-1}(l_p,\ldots,l_1)$.  That is, the first row of $G$
provides the coefficients of the factor $l(z)$ normalized so that
$l_0=-1$. 
Similarly, one finds that
$Re_1=-u_0^{-1}(u_p,u_{p-1},\ldots,u_1)^T$, and $e_{p-n_++1}^TR=-u_{n_+}
e_1^TU_0^{-1}$.  That is, the first column of $R$ provides the
coefficients of the factor $u(x)$ normalized so that $u_0=-1$. In order
to determine the normalizing constant $w$ such that
$a(z)=u(z)wl(z^{-1})$, it is sufficient to impose the condition $
u_{n_+} w l_0=a_{n_+}$ so that we can choose $w=-a_{n_+}/u_{n_+}$.

This argument provides the following algorithm to compute the
coefficients of $l(x)$ and of $u(x)$ such that
$a(z)=u(z)wl(1/z)$, where $u_0=l_0=-1$:

\begin{enumerate}
	\item Assemble the matrices $A_{-1}, A_0, A_1$. 
	\item Determine $R$ and $G$ that solve \eqref{eq:mateq} using cyclic reduction. 
	\item Compute $\hat u=Re_1$, set $u=(-1,\hat u_p,\ldots,\hat u_1)$ 
	  and  $\hat v=e_1^TG$, set $l=(-1,\hat v_p,\ldots, \hat v_1)$. 
	\item Set $w=-a_{n_+}/u_{n_+}$
\end{enumerate}

Observe that the above algorithm can be easily modified in order to
compute, for a given $q$, the first $q$ coefficients of the triangular
Toeplitz matrices $\mathcal U^{-1}$ ad $\mathcal L^{-1}$ such that
\[
\mathcal A^{-1}=\frac 1w\mathcal L^{-1}\mathcal U^{-1}.
\] 
In fact, the first $p$ coefficients, are given by 
\[\begin{split}
&L_0^{-1}e_1=-\frac 1{l_{n_-}}Ge_{p-n_-+1},\\
&e_1^TU_0^{-1}=-\frac1{u_{n_+}}e_{p-n_++1}^TR.
\end{split}
\]
While the remaining coefficients can be computed by means of the
Sieveking-Kung algorithm described in the previous section.

The method described in this section requires the computation of the
solutions $G$ and $R$ of equation \eqref{eq:mateq}. One of the most
effective methods to perform this computation is the Cyclic Reduction
(CR) algorithm. We refer the reader to \cite{bini2009cyclic} for a review
of this method, and to \cite{BM:simax} for the analysis of the
specific structural and computational properties of the matrices
generated in this way. Here we provide a short outline of the
algorithm, applied to the equations \eqref{eq:mateq} which we rewrite
in the form $AG^2+BG+C=0$ and $R^2C+RB+A=0$, respectively.  The
algorithm CR computes the following matrix sequences
\begin{equation}\label{eq:cr}
\begin{split}
&B^{(k+1)}=B^{(k)}-A^{(k)}S^{(k)}C^{(k)}-C^{(k)}S^{(k)}B^{(k)},\quad
S^{(k)}=(B^{(k)})^{-1},\\ &A^{(k+1)}=-A^{(k)}S^{(k)}A^{(k)},\quad
C^{(k+1)}=-C^{(k)}S^{(k)}C^{(k)},\\ &\widehat B^{(k+1)}=\widehat
B^{(k)}-C^{(k)}(B^{(k)})^{-1}A^{(k)},\quad \widetilde
B^{(k+1)}=\widetilde B^{(k)}-A^{(k)}(B^{(k)})^{-1}C^{(k)}.
\end{split}
\end{equation}

It is proved that under mild conditions the sequences can be computed
with no breakdown and that $\lim_{k} -A(\widehat B^{k)})^{-1} =R$,
$\lim_{k}-\widetilde (B^{(k)})^{-1}C=G$. More precisely, the following relations hold
\[\begin{split}
&G = -(\widetilde B^{(k)})^{-1}C-  \widetilde (B^{(k)})^{-1} A^{(k)}G^{2^k}\\
&R = -A(\widehat B^{(k)})^{-1}-  R^{2^k}C^{(k)}\widehat (B^{(k)})^{-1}
\end{split}
\]
and it can be proved that $\|(\widehat B^{(k)})^{-1}\|$ and $\|(\widetilde
B^{(k)})^{-1}\|$ are uniformly bounded by a constant and that
$A^{(k)}$, $C^{(k)}$ converge double exponentially to zero.  Since the
spectral radii of $R$ and of $G$ are less than 1, this fact implies
that convergence is quadratic. Moreover, the approximation errors
given by the matrices $ (\widetilde B^{(k)})^{-1} A^{(k)}G^{2^k}$ and
$ R^{2^k}C^{(k)}(\widehat B^{(k)})^{-1}$ is explicitely known in a
first order error analysis. In fact the matrices $ (\widetilde
B^{(k)})^{-1}$, $(\widehat B^{(k)})^{-1}$, $A^{(k)}$ and $C^{(k)}$
are explicitely computed by the algorithm and $G$ is approximated. 
This fact allows us to implement effectively the Wiener-Hopf computation required in the inversion
procedure described in Algorithm~1
of Section \ref{sec:inv}.

The cost of Cyclic Reduction is $O(p^3)$ arithmetic operations per step. 
In \cite{BM:simax} it is shown that all the above matrix sequences are
formed by matrices having displacement rank bounded by small
constants. This fact enables one to implement the above equation with
a linear cost, up to logarithmic factors,  by means of FFT.

\subsubsection{A different approach}
Another approach to compute the factor $l$ and $u$ relies on the
following property \cite{bibo}.

\begin{theorem} Let $a(z)^{-1}=h(z)=\sum_{i=-\infty}^{\infty}h_iz^i$.
	Define the Toeplitz matrix of size $q\ge \max(m,n)$
	$T_q=(h_{j-i})$. Then, $T_q$ is invertible and its last row and column
	define the coefficient vectors of $l(z)$ and $u(z)$, respectively
	up to a normalization constant.
\end{theorem}

\begin{proof} The relation $a(z)^{-1}=l^{-1}(z^{-1})u^{-1}(z)$ can
be rewritten in matrix form as
\[
(h_{j-i})=T(u^{-1})T(l^{-1})^T.
\]
Multiply to the right by the infinite vector obtained by completing
$(l_{0},\ldots,l_{q-1})$ with zeros. Since the product of $T(l^{-1})^T$ with the latter is a
vector with all null components except  the first one, equal to 1, considering
$q$ components of the result yields
\[
T_q(h)(l_{0},\ldots,l_{q-1})^T=(T(u^{-1}))e_1=u_0^{-1} e_q
\]
whence we deduce that $(l_0,\ldots,l_{q-1})^T=T_q(h)^{-1}u_0^{-1}e_q$.
Similarly we do for the last row.
\end{proof}

This property is at the basis of the following computations

\begin{enumerate}
	\item Set $q=\max(m,n)$ compute $h_{i}$ for $i=-q,q$ such that $h(z)a(z)=1$ by means of evaluation/interpolation.
	
	\item Form $T_q(h)=(h_{j-i})_{i,j=1,q}$ and compute last row and last column of $T_q(h)^{-1}$.
\end{enumerate}

This algorithm may require a large number of interpolation points when
$a(z)$ has some zero of modulus close to 1, in the process of
evaluation/interpolation.

\subsubsection{Yet another approach}
The same property provides a third algorithm for computing $l(z)$
and $u(z)$ which relies on a different computation of $h_i$,
$i=-q,\ldots, q$. The idea is described below

Consider the equation
\[
a(z)h(z)=1.
\]
Multiply it by $a(-z)$ and, since $a(-z)a(z)=a_1(z^2)$, for a
polynomial $a_1(z)$, get
\[
a_1(z^2)h(z)=a(-z).
\]
Repeating the procedure $k$ times yields
\[
a_k(z^{2^k})h(z)=a_{k-1}(-z^{2^{k-1}})\cdots a_1(-z^2)a(-z).
\]
If $a(z)$ has roots of modulus different from 1, then $a_k(z)$ quickly
converges to either a constant or a scalar multiple of $z$, since its zeros are the $2^k$ powers of the
zeros of $a(z)$. In this case, $h(z)$ can be computed by means of a
product of polynomials with the same degree (independent of the
iterations).

\subsubsection{Newton's iteration}
Newton's iteration can be applied to the nonlinear system
$a(z)=u(z)l(z^{-1})$ where the unknowns are the coefficients of the
polynomials $u(z)$ and $l(z)$. The Jacobian matrix has a particular
structure given in terms of displacement rank which can be exploited
to implement Newton's iteration at a low cost. Details in this regard
are given in the papers \cite{boett} and \cite{boett1}.

\bibliographystyle{abbrv}
\bibliography{bibliography}
  
\end{document}